\documentclass[12pt]{amsart} %the usual amsart package default fontsize is unclear, 12pt does not give 12pt! so use the following line to get the 12pt fontsize
\usepackage[fontsize=12pt]{scrextend}
\usepackage[utf8]{inputenc}
\usepackage{amsmath}
\usepackage{amsfonts}
\usepackage{amssymb}
\usepackage{tikz-cd}
\usepackage{hyperref}
\usepackage[english]{babel}
\usepackage{amsthm}
\usepackage{mathtools}
\usepackage{hyperref}
\usepackage{enumitem}
\usepackage[letterpaper, left=3cm, right=3cm, top = 2cm, bottom = 2cm ]{geometry}
\theoremstyle{definition} %this takes care of italicizing of the body in theorems lemma and proof body {plain} makes it italicized and is default the other is {remark} \begin proof comes with the tombstone default.
\newtheorem*{remark}{Remark}
\newtheorem{lem}{Lemma}[section]
\newtheorem{prop}[lem]{Proposition}
\newtheorem{coro}[lem]{Corollary}
\newtheorem{theorem}[lem]{Theorem}
\newtheorem{defn}[lem]{Definition}
\newtheorem{de}{Definition}[section]
\newtheorem{lemma}[de]{Lemma}
\newtheorem{eg}[de]{Example}

\makeatletter% For roman numerals

\newcommand{\Rmnum}[1]{\expandafter\@slowromancap\romannumeral #1@}
\makeatother

\begin{document}

\title{Character on a homogeneous space }

\author{A.J. Parameswaran, Amith Shastri K.}
\address{School of Mathematics, Tata Institute of Fundamental Research,  
1 Homi Bhabha Road, Colaba, Mumbai, India}
\email{param@math.tifr.res.in}
\email{shastriq@math.tifr.res.in}

\begin{abstract}
In this paper we look at the notion of cohomological triviality of fibrations of homogeneous spaces of affine algebraic groups defined over $\mathbb{C}$ and use topological methods, primarily the theory of covering spaces. This is made possible because of the structure theory of affine algebraic groups. Further, we generalize our results for arbitrary connected algebraic groups and their homogeneous spaces. As an application of our methods, we give a structure result for quasi- reductive algebraic groups(i.e groups whose unipotent radical is trivial), upto isogeny.
\end{abstract}
\maketitle

\section{Introduction}
We will be working over the field of complex numbers throughout this paper. A linear algebraic group $G$ acting linearly on a vector space $V$, so that there is an open dense orbit $U$, is known as a prehomogeneous representation  and has been extensively studied by Sato and Kimura in \cite{Sato}. Further, $U\approx \frac{G}{H}$ is a homogeneous space, where $H$ is a closed subgroup of $G$. In \cite{Damon}, Damon considers a special kind of prehomogeneous representations, in which the complement of $U$ is actually a divisor, the defining polynomial $f$ of the divisor is a relative invariant and hence is homogeneous ( c.f. corollary 2.7 in \cite{Kimura}) . We note that, when $G$ is reductive the following are equivalent 
\begin{enumerate}
\item $U^{c}$ is a hypersurface.
\item $H$ is reductive. 
\item $\frac{G}{H}$ is affine.
\end{enumerate} 
 The equivalence of  2 and 3 is known as Matsushima's theorem see \cite{Rich}. For the proof of the above equivalence see \cite{Kimura} Theorem 2.28 page 43. This function $f$ is defined on all of $V$, and in particular restricting to $U$ we have $f:U\to \mathbb{C}^*$. From the point of view of singularity theory, which Damon is concerned with, the complement of $U$ defines a non- isolated singularity and $f:\frac{G}{H}\to\mathbb{C}^*$ is a global Milnor fibration, with fibres also being homogeneous spaces.
 For the study of the Milnor fibration Damon defines (rational) cohomological triviality of fibrations and obtains a very simple numerical criterion for fibrations over the base $S^1$  to satisfy cohomological triviality.\par
In this paper, however, we are not concerned with the topology of the non- isolated singularities. We are interested in the characterization of cohomologoical triviality of  fibrations over $\mathbb{C}^*$ considering the general setting of homogeneous spaces of algebraic groups, instead of the set up of prehomogeneous spaces with the complement a divisor as in  \cite{Damon}. \par

In section \ref{Prelims} we define the equivalent notion of cohomological triviality of fibrations as in \cite{Damon} (the equivalence of both is also proved in the same article). We will be working with singular cohomology with rational coefficients throughout and will omit the coefficients. We then prove some preliminary topological results and look at examples of fibrations which are cohomologically trivial and also examples of fibrations which are not cohomologically trivial.

In section \ref{Ctr} we come to the main results of the paper, however instead of considering prehomogeneous representations of an affine algebraic group $G$ we consider homogeneous spaces of connected affine algebraic groups with a nowhere vanishing non- constant  function to $\mathbb{C}$. This function in turn gives rise to a character, $\mathcal{X}$, of $G$. This character in turn gives rise to an auxiliary fibration which we first analyse and we show that the fibration, which is also an exact sequence of algebraic groups,
\begin{center}
\begin{tikzcd}
S^0\arrow{r} & G\arrow{r}{\mathcal{X}_0} & \mathbb{C}^*
\end{tikzcd}
\end{center}
is cohomologically trivial, where $S^0$ is the identity component of $ker(\mathcal{X})$ and the kernel of $\mathcal{X}_0$. \par
Before starting the analysis of fibrations of affine groups we make a simplifying assumption that the group $G$ is reductive. This assumption is in no loss of generality,from the point of topology, the unipotent radical $G_u$ is contractible thus $G$ and $\frac{G}{G_u}$ have the same homotopy type and from the point of algebraic groups any character of $G$ restricts as the trivial character on $G_u$. From this auxiliary analysis we will show that when the kernel of the associated character is connected then the fibration of homogeneous spaces 
\begin{center}
\begin{tikzcd}
\frac{S}{H}\arrow{r} & \frac{G}{H}\arrow{r}{\mathcal{X}} & \mathbb{C}^*
\end{tikzcd}
\end{center} 
is cohomologically trivial.\par
From the above we will deduce that the fibration
\begin{center}
\begin{tikzcd}
\frac{S}{H} \arrow{r} & \frac{G}{H}\arrow{r} {\mathcal{X}} & \mathbb{C}^*
\end{tikzcd}
\end{center}
 is cohomologically trivial if and only if the natural map$ \frac{G}{S^0 \cap H} \to \frac{G}{H}$ induces an isomorphism $H^*\big(\frac{G}{H}\big) \approx H^* \big( \frac{G}{S^0 \cap H} \big)$. Our proof, in particular, shows that the monodromy group is finite and hence is semisimple.\par
In section \ref{con} we construct a class of fibrations with connected fibres, which are not cohomologically trivial as a converse of sorts for the results in section 3.
In section \ref{arb} we move to the very general setting of homogeneous spaces of algebraic groups, not necessarily affine. We prove some structure theorems for these groups viz. the existence of an unipotent radical and maximal central torus. As before we quotient out the unipotent radical and we will show that our results hold in this general setting for algebraic groups with trivial unipotent radical which we call quasi- reductive group. 
\section{Preliminaries}\label{Prelims}
We will use the following equivalent definition of cohomological triviality of fibrations:
\begin{de}
A fibration $F\hookrightarrow E\rightarrow B$ is said to be rationally cohomologically trivial if it satisfies the K\"unneth formula, i.e. for each $k$
\begin{center}
$H^k(E; \mathbb{Q})\cong$ $\mathop{\oplus} \limits_{p+q=k} H^p(F; \mathbb{Q})\otimes H^q(B; \mathbb{Q})$,\\
\end{center}
the isomorphism is of graded vector spaces.\\
This implies, 
\begin{center}
$b_k(E)$= $\sum\limits_{p+q=k} b_p(F)\cdot b_q(B)$.
\end{center}
for all $k$, where $b_k$ is the kth Betti number.
\end{de}
Following \cite{Damon} we can also define (rational) cohomological triviality of a fibration when the base is $\mathbb{C}^*$as follows.
Let $F\hookrightarrow E\xrightarrow{\pi} \mathbb{C}^*$ be a fibration, let $\alpha(t)= e^{2\pi it}$ be a loop at 1, now consider the monodromy map $\sigma_1$, the lift of $\alpha$ to a family of homeomorphisms $\sigma_t:F\rightarrow F_t= \pi^{-1}\alpha(t)$.
\begin{lemma}\label{le}
A fibration $F\hookrightarrow E\xrightarrow{\pi}\mathbb{C}^*$ is (rationally) cohomologically trivial if $\sigma_1^{\ast}: H^\ast(F;\mathbb{Q})\rightarrow H^\ast(F;\mathbb{Q})$ is the identity.
\end{lemma}
\begin{proof}\renewcommand{\qedsymbol}{$\blacksquare$}
By Proposition 1.8 of \cite{Damon} the definition we have given in the beginning and for the special case when the base space is $\mathbb{C}^*$ coincide.
\end{proof}
Observe that if $E$ is connected and the fibre $F$ is not connected, then the fibration
$F\hookrightarrow E\rightarrow B$ is not cohomologically trivial. This is clear by looking at  $b_0$, the zeroth Betti number. This leads us to make a global assumption that the fibres of any fibration we are considering to be connected.\par

We now give another formulation of cohomological triviality of a fibration $F\rightarrow E\rightarrow B$ as degeneration of the Leray- Serre spectral sequence. Proposition 5.5 in \cite{Mcc} (page 139) states that the $E_2^{r,s}$ term in the Leray- Serre spectral sequence is given by the tensor product $H^*(B; \mathbb{Q})\otimes H^*(F;\mathbb{Q})$ under the conditions that $F$ and $B$ are of finite type and the system of local coefficients on $B$ determined by the fibre $F$ is simple, and the spectral sequence converges to $H^*(E;\mathbb{Q})$. Thus if the spectral sequence degenerates at $E_2$: 
$$H^n(E;\mathbb{Q})=\mathop{\oplus}\limits_{r+s=n}E_\infty^{r,s}=\mathop{\oplus}\limits_{r+s=n} H^r(B; \mathbb{Q})\otimes H^s(F;\mathbb{Q})$$
Conversely, suppose that $F$ and $B$ are of finite type and the system of local coefficients on $B$ determined by the fibre $F$ is simple and the fibration is cohomologically trivial. Then by proposition 5.5 in \cite{Mcc}, we have that $E_2^{r,s}= H^r(B; \mathbb{Q})\otimes H^s(F;\mathbb{Q})$ and the spectral sequence converges to $H^*(E;\mathbb{Q})=\mathop{\oplus}\limits_{r+s=n} H^r(B; \mathbb{Q})\otimes H^s(F;\mathbb{Q})$. Thus the Leray- Serre spectral sequence degenerates at $E_2$.

 If the base is $\mathbb{C}^*$ we have,
\begin{lemma}\label{lem1}
Let $F\rightarrow E\xrightarrow{p}\mathbb{C}^*$ be a fibration, such that the fibre $F$ is not connected, and the total space $E$ and the base be path connected . Then there exists a connected covering $B\rightarrow\mathbb{C}^*$ and a lift $\tilde{p}: E\to B$,  such that the fibre bundle $F_0\rightarrow E\xrightarrow{\tilde{p}}B$ has connected fibre $F_0$.
\end{lemma}
\begin{proof}\renewcommand{\qedsymbol}{$\blacksquare$}
Since $E$ and $\mathbb{C}^*$ are path connected the long exact homotopy sequence of the fibration $F\rightarrow E\xrightarrow{p}B$ reduces to,
\begin{center}
\begin{tikzcd} [cramped, sep=small] 
...\pi_{1}(F,f)\arrow[r] & \pi_1(E,e)\arrow[swap]{r}{p_{*}} & \pi_1(\mathbb{C}^*,1)\arrow[r] & \pi_0(F,f)\arrow[r] & 1\\
\end{tikzcd}
\end {center}
This sequence shows, in particular, that $p_*$ is surjective if and only if the fibre $F$ is connected.\par
Let the number of path components of $F$ be $n$, then 
$Im(\pi_1(E))$ is a subgroup of $\pi_1(\mathbb{C}^*)\approx \mathbb{Z}$, of index $n$. Hence we have the following commutative diagram,
\begin{center}
\begin{tikzcd}
& & \mathbb{C}^*\arrow[ d,xshift=1mm,"z\mapsto z^n"]\\
F\arrow[swap]{r} & E\arrow[swap]{r}{p}\arrow[ur, dotted, "\tilde{p}"] & \mathbb{C}^*
\end{tikzcd}
\end{center}
where the dotted arrow is the lift of map $p$, which exists as the map
$\mathbb{C}^*\rightarrow\mathbb{C}^*$ is a covering map. Now the fibre bundle $E\xrightarrow{\tilde{p}}\mathbb{C}^*$ has connected fibre.
\end{proof}

Let us now look at some examples of fibrations and check cohomolgical triviality.
\begin{eg}\label{eg1}
The Hopf fibration
\begin{center}
$S^1\longrightarrow S^3\longrightarrow S^2$
\end{center}
is not cohomologically trivial, for the term $H^2(S^3)$ is zero, 
whereas \\$H^2(S^2)\otimes H^0(S^1)= \mathbb{Z}$.
\end{eg}
We now give an example of a fibration which is cohomologically trivial
\begin{eg}\label{eg2}
For any $n$, the exact sequence of algebraic groups,
\begin{center}
$1\rightarrow SL_n\rightarrow GL_n\xrightarrow{det}\mathbb{C}^*\rightarrow 1$
\end{center}
is cohomologically trivial, where $det$ is the determinant character. For this the cohomologies of $GL_n$ and $SL_n$ are  exterior algebra on odd degree generators with the cohomology of  $SL_n$ beginning at degree 3 and $GL_n$ beginning at degree 1. For details refer \cite{Toda} and \cite{Damon}.
Furthermore, note that $GL_n$ is a product of $SL_n$ and $\mathbb{C}^*$ as manifolds but not as groups.
\end{eg}
\begin{eg}\label{eg3}
Similarly by looking at the cohomology algebra, the fibration
\begin{center}
$\frac{SL_n}{SO_n}\rightarrow\frac{GL_n}{O_n}\rightarrow\mathbb{C}^*$
\end{center}
is not cohomologically trivial for $n= 2m$ and is cohomologically trivial for $n= 2m+1$. We refer \cite{Damon} and \cite{Toda} for the details.
\end{eg}
\begin{eg}\label{eg4}
But, the fibration $\frac{SL_n}{SO_n}\rightarrow\frac{GL_n}{SO_n}\rightarrow\mathbb{C}^*$ is cohomologically trivial for all $n$.  We refer to \cite{Toda} for the details.
\end{eg}

We shall now prove some preliminary results,

\begin{lemma}\label{lem2}
Consider the following, 
$H\hookrightarrow G\rightarrow \frac{G}{H}$,  where $G$ is an affine algebraic group and $H$ is a closed subgroup of $G$. Then 
$H\rightarrow G\rightarrow\frac{G}{H}$ is a fibre bundle.
\end{lemma}
\begin{proof}\renewcommand{\qedsymbol}{$\blacksquare$}
See page 105, Example 2.1.1.4 (ii) for a proof in \cite{Schmitt}.
\end{proof}
We note that in the above lemma the group $G$ being affine is not necessary.
\begin{lemma}\label{lem5}
Let $G$ be an algebraic group and $\mathcal{X}: G\longrightarrow\mathbb{C}^*$ be a non constant morphism of algebraic varieties, taking identity to identity. Then $\mathcal{X}$ is a character.
\end{lemma}
\begin{proof}\renewcommand{\qedsymbol}{$\blacksquare$}
This is a particular case of theorem 3 in \cite{Rosen}. For a modern proof we refer to corollary 1.2 in the  unpublished notes of Brian Conrad \cite{Bcnrd1}, where this result is referred to as Rosenlicht unit theorem.
\end{proof}

\section{Cohomological triviality of characters }\label{Ctr}
Let $\mathcal{X}$ be any non trivial map from $ G$ onto $\mathbb{C}^* $ then by the lemma \ref{lem5}, $\mathcal{X}$ is a non trivial character on $G$ . We note that there can be no non- constant algebraic maps from an unipotent group $U$ to $\mathbb{C}^*$, in particular the unipotent radical $G_u$ of $G$ is in the kernel of $\mathcal{X}$ and $\mathcal{X}$ factors through $\frac{G}{G_u}$. Furthermore topologically $G_u$ is an affine space and hence is contractible thus $G$ and $\frac{G}{G_u}$ are of the same homotopy type. Thus it suffices to look at reductive groups for the question of cohomological triviality. We record this as a lemma.\par
\begin{lem}\label{lem3.1}
Let $G$ be an arbitrary linear algebraic group and $\mathcal{X}:G\rightarrow\mathbb{C}^*$ be a non-trivial character with connected kernel $S$.
Then the fibration
\begin{center}
$S\rightarrow G\rightarrow\mathbb{C^*}$
\end{center}
is homotopic to the fibration
\begin{center}
\begin{tikzcd}
\frac{S}{G_u}\arrow{r}{} & \frac{G}{G_u}\arrow{r}{\mathcal{X}} & \mathbb{C}^*
\end{tikzcd}
\end{center}
\hfill $\blacksquare$
\end{lem}

Throughout the section we will assume that $G$ is a connected reductive group by the reasoning of lemma \ref{lem3.1}. 
\begin{prop} \label{prop1}
Let $G$ be a reductive group, and $\mathcal{X}: G\longrightarrow\mathbb{C}^*$ be a non trivial character with the kernel $S$, which we assume to be connected. Then the fibration
\begin{center}
$1\rightarrow S\rightarrow G\rightarrow\mathbb{C}^*\rightarrow 1$
\end{center} 
is cohomologically trivial.
\end{prop}
\begin{proof}\renewcommand{\qedsymbol}{$\blacksquare$}
By lemma \ref{lem2} the sequence $1\rightarrow S\rightarrow G\rightarrow\mathbb{C}^*\rightarrow 1$ is a fibre bundle. 
Since $G$ is reductive there is a central $\mathbb{C}^*$ which surjects onto $\mathbb{C}^*$ under $\mathcal{X}$. This follows from the fact that a character of a reductive group is identically zero on the derived group and hence factors through $\frac{G}{D(G)}$, where $D(G)$ is the derived group of $G$ and $\frac{G}{D(G)}$ is isogenous with the identity component of the center $Z(G)$. Thus if the character is non trivial we can find a central $\mathbb{C}^*$ in $G$ that surjects onto $\mathbb{C}^*$. This $\mathbb{C}^*$ can be seen to be as follows- consider the restriction map $\mathcal{X}: Z(G)^0\rightarrow \frac{Z(G)^0}{ker(\mathcal{X})^0}\xrightarrow{}\frac{Z(G)^0}{ker(\mathcal{X})}$ the last two terms are isomorphic to $\mathbb{C}^*$ and hence the map is given by $z\mapsto z^d$ now choose any splitting of $\mathcal{X}: Z(G)^0\rightarrow \frac{Z(G)^0}{ker(\mathcal{X})}$. i.e.
\begin{center}
\begin{tikzcd}
S \arrow{r}{} 
& G \arrow{r}{\mathcal{X}} 
& \mathbb{C}^*\\
& \mathbb{C}^* \arrow[hookrightarrow]{u}{}  \arrow[swap]{ur}{\mathcal{X}|_{\mathbb{C}^*}}
\end{tikzcd}
\end{center}
The kernel of the character $\mathcal{X}|_{\mathbb{C}^*}:\mathbb{C}^*\longrightarrow\mathbb{C}^*$ is a finite cyclic group, say $\Gamma$, 
which induces a fibre product $\tilde{G}$ in  $G\times \mathbb{C}^*$
\begin{center}
\begin{tikzcd}
S \arrow{r}{}  \arrow [swap]{d}{}  
& \tilde G \arrow{r}{}  \arrow [swap]{d}{\pi_1} \arrow[dr, phantom, "\square"] %phantom is phantom! invisible arrows!
& \mathbb{C}^* \arrow{d}{\mathcal{X}|_{\mathbb{C}^*}} \\
S \arrow{r}{}  
& G   \arrow [swap]{r}{\mathcal{X}} 
&\mathbb{C}^* 
\end{tikzcd}
\end{center}
The cover $\tilde{G}$, contained  in $G\times \mathbb{C}^*$, is connected (follows by looking at the long exact sequence of homotopy groups associated to the fibration and since the fibre $S$ and the base space are connected) and abstractly isomorphic to $S\times \mathbb{C}^*$ by the map $f: S\times \mathbb{C}^*\longrightarrow\tilde{G}$ given by $(s,t)\longmapsto (st,t)$. This map is also a group homomorphism, which can be seen as follows
$(s_1,t_1)(s_2,t_2) \mapsto (s_1t_1s_2t_2,t_1t_2)$ and since the chosen $\mathbb{C}^*$ is central in $G$ we have,
$(s_1t_1s_2t_2,t_1t_2)= (s_1s_2t_1t_2,t_1t_2)$ and the other way, $(s_1,t_1)(s_2,t_2)= (s_1s_2,t_1t_2) \mapsto (s_1s_2t_1t_2,t_1t_2)$.

Thus the first projection $\pi_1$:$\tilde{G}\longrightarrow G$ can
be thought of as twisted by the isomorphism $f$, and hence by the identification by $f$ we can see that the upper fibration, $S\to \bar{G}\to \mathbb{C}^*$ is not just
cohomologically trivial but is actually a trivial fibration. 
Now consider any loop in $\mathbb{C}^*$, then its lift to $\mathbb{C}^*$ is a path from 0 to $\zeta$ where $\zeta$ is a root of unity (corresponding an element of the finite cyclic group $\Gamma$). Since $\zeta\in S$  we have a self map of $S$ given by translation by $\zeta$ and since $S$ is path connected we have that the translations are null-homotopic, therefore the monodromy of the fibration $\tilde{G}\to \mathbb{C}^*$ is trivial. Thus by commutativity of the second diagram the fibration $G\to\mathbb{C}^*$ is also cohomologically trivial by lemma \ref{le}
\end{proof}
This gives another proof for the cohomological triviality of the fibration given in example 2.5 of section 2.\par
We can see the cohomological triviality of the fibration $S\to G\to \mathbb{C}^*$  by a Wang sequence argument as given in \cite{Damon}, considering a fibration of maximal compact subgroups and exploiting the Hopf algebra structure of the cohomology.

\begin{defn}
A character $\mathcal{X}$ is said to be split if $G= S\times\mathbb{C}^*$, where $S$ is the kernel of $\mathcal{X}$. A character $\mathcal{X}$ is said to be $\textit{quasi- split}$ if $\mathcal{X}$ is split after a finite covering.
\end{defn}
As a corollary to the method of proof in the proposition \ref{prop1} we have
\begin{coro}\label{coro}
A character $\mathcal{X}: G\rightarrow\mathbb{C}^*$ is split if and only if $\mathcal{X}|_{Z(G)^0}$ is primitive.
\end{coro}
\begin{proof}\renewcommand{\qedsymbol}{$\blacksquare$}
 Note that $Z(G)^0= Z(S)^0\times \mathbb{C}^*$.\par
Now suppose that the character $\mathcal{X}$ is split, then we have an isomorphic image of $\mathbb{C}^*$, say $G_m$, in $Z(G)^0$, i.e there is a section of $\mathcal{X}$ restricted to $Z(G)^0$. Thus $\frac{Z(G)^0}{ker(\mathcal{X}|_{Z(G)^0})}\approx \mathbb{C}^*$, i.e. $\mathcal{X}|_{Z(G)^0}$ is primitive.\par
Suppose that the restriction of $\mathcal{X}$ to $Z(G)^0$ is primitive, i.e. the kernel $S_Z$ is connected and hence, by corollary in \cite{Borel} page 118, $S_Z$ is a torus and a direct factor in $Z(G)^0$. Thus $Z(G)^0= S_Z\times G_m$ and since $ker(\mathcal{X})$ is connected ,this $G_m$ maps isomorphically onto $\mathbb{C}^*$, and by the proof of proposition \ref{prop1} we have that the fibre product $\tilde{G}\approx S\times \mathbb{C}^*$ is isomorphic to $G$. By the proposition we need to just prove the injectivity of the composite of $f$ with the projection map $\pi_1$. Let $(s_1,t_1)$ and $(s_2,t_2)$ be two arbitrary elements of $S\times G_m$ which map to the same element in $G$, i.e. $s_1t_1=s_2t_2$. Now, $\mathcal{X}(s_1t_1) =\mathcal{X}(s_2t_2)$ and since $\mathcal{X}$ is a homomorphism  and $s_1$, $s_2$ are in the kernel of $\mathcal{X}$ we have that $\mathcal{X}(t_1)=\mathcal{X}(t_2)$, and since $\mathcal{X}$ restricted to $G_m$ is an isomorphism we have that $t_1=t_2$, from this it follows that $s_1=s_2$ and thus the composite is injective and hence the character is split.
\end{proof}
This corollary, in particular, shows that $GL_n$ with the determinant character is not split.

\begin{coro}
Consider a fibration of homogeneous spaces of a reductive group $G$
\begin{center}
$\frac{S}{H}\rightarrow\frac{G}{H}\rightarrow\mathbb{C}^*$.
\end{center}
Suppose that the homogeneous space $\frac{S}{H}$ is connected and that the subgroup $H$ is normal in $G$, then the fibration is cohomologically trivial.
\end{coro}
\begin{proof}\renewcommand{\qedsymbol}{$\blacksquare$}
Since $H$ is normal $\frac{G}{H}$ and $\frac{S}{H}$ are groups and by the propsition \ref{prop1} we are done.
\end{proof}

We note that we are not considering all homogeneous spaces of a reductive group $G$,  for example we will not consider $H$ to be a Borel or a parabolic subgroup of $G$, for the quotient $\frac{G}{H}$ would then be projective and thus has no non- trivial morphisms to any affine variety.\par
Let $\mathcal{X}:\frac{G}{H}\rightarrow\mathbb{C}^*$ be a nowhere vanishing function on a homogeneous space of an algebraic group $G$, with $H$ mapping to 1. Then we have a natural map $\mathcal{X}': G\rightarrow\mathbb{C}^*$ by composing with the quotient map. We will refer the map $\mathcal{X}$ on $\frac{G}{H}$ as a character. We will call the the fibration $S\rightarrow G\xrightarrow{\mathcal{X}} \mathbb{C}^*$ to be the associated fibration of algebraic groups, where $S$ is the kernel of $\mathcal{X}$. We will first show that if the kernel of the associated fibration is connected then the associated fibration as well as the fibration of homogeneous spaces of affine algebraic groups is cohomologically trivial.\par

We shall now prove cohomological triviality of fibrations of homogeneous spaces when the kernel $S$ is connected.
\begin{theorem}\label{thm1}
Consider  a fibration of homogeneous spaces of  a connected reductive group $G$: 
\begin{center}
$\frac{S}{H}\longrightarrow\frac{G}{H}\longrightarrow\mathbb{C}^*$
\end{center}
Suppose that the kernel $S$ of the associated character is connected, then the fibration is cohomologically trivial.
\end{theorem}
\begin{proof}\renewcommand{\qedsymbol}{$\blacksquare$}
 Consider the fibration
$1\longrightarrow\frac{S}{H}\longrightarrow\frac{G}{H}\longrightarrow\mathbb{C}^*\longrightarrow 1$, by composing with the natural projection map from $G$ we get a map
from $G$ to $\mathbb{C}^*$ 

\begin{center}
\begin{tikzcd}
\frac{G}{H} \arrow{r}{\mathcal{X}}  & \mathbb{C}^*\\
     G \arrow{u}{\pi}  \arrow[swap]{ur}{\mathcal{X}'}
\end{tikzcd}
\end{center}
which we will call $\mathcal{X}'$. This is a nowhere vanishing map 
from $G$ onto $\mathbb{C}^*$ and thus is a non trivial character of $G$ by lemma \ref{lem5} (note that the identity element $e$ of $G$ is mapped to 1, as the subgroup $H$ is mapped identically to 1). We will show that the lift of a loop in $\mathbb{C}^*$ is homotopic to identity and hence, the fibration is cohomologically trivial. 
Now consider the fibration of homogeneous space of a closed subgroup $H$ , which is not necessarily connected, contained in $S$.

 This gives rise to a commutative diagram as follows:
\begin{center}
\begin{tikzcd}
\frac{S}{H} \arrow{r}{}  \arrow [swap]{d}{}  & \frac{\tilde G}{H} \arrow{r}{}  \arrow [swap]{d}{\pi_1} \arrow[dr, phantom, "\square"] 
& \mathbb{C}^*\arrow{d}{\mathcal{X}} \\
\frac{S}{H} \arrow{r}{}  & \frac{G}{H}   \arrow [swap]{r}{\mathcal{X}} &\mathbb{C}^* 
\end{tikzcd}
\end{center}
As in proposition \ref{prop1}, $\frac{\tilde{G}}{H}$ is isomorphic to 
$\frac{S}{H}\times\mathbb{C}^*$ induced by the map from $S\times\mathbb{C}^*\longrightarrow \tilde{G}$ i.e. $(\bar{x},t)\mapsto(\bar{xt},t)$ to see that the map defined as above is well defined note that $\bar{x}=xh$ for some $h$ in $H$, thus $(\bar{x},t)=(xh,t)\mapsto(xht,t)=(xth,t)$ (as $\mathbb{C}^*$ is central)which is same as $(\bar{xt},t)$. Thus the upper fibration $\frac{S}{H}\to \frac{\tilde{G}}{H}\to \mathbb{C}^*$ is trivial.
To show that the lift of a loop is homotopic to identity, for the homogeneous spaces we follow a similar argument as above for the case of the group and noting that a loop in the lower $\mathbb{C}^*$ ends at $\zeta$ which is in $S$ and hence in $\frac{S}{H}$ which gives rise to a self map of $\frac{S}{H}$ given by translation by $\zeta$. But translations in an homogeneous space of a connected algebraic group are isotopic to identity (since $\zeta \in S$, there is a path from $\zeta$ to 1 in $S$ and this gives rise to automorphisms of $\frac{S}{H}$ which are homotopic to identity). And thus in the upper level the loop is homotopic to identity. Thus by the commutativity of the diagram the loop is homotopic to identity in the lower level. Thus by lemma \ref{le} the fibration $\frac{S}{H}\to\frac{G}{H}\to \mathbb{C}^*$ is cohomologically trivial.
\end{proof}
\begin{remark}
Note that, we have shown in proposition \ref{prop1} (and for homogeneous spaces in theorem \ref{thm1}) that for a reductive group $G$ and a non trivial, non constant map $\mathcal{X}$ to $\mathbb{C}^*$ with a connected kernel $S$ is a product $S\times \mathbb{C}^*$ (a product $\frac{S}{H}\times \mathbb{C}^*$ for homogeneous spaces) after a finite cover! This in particular shows that the monodromy is semisimple.
\end{remark}
We will now give a  criterion for cohomological triviality when the kernel $S$ of the associated character is not necessarily connected.
\begin{coro}\label{cor1}
Consider  a fibration of homogeneous spaces of  a connected reductive group $G$:
\begin{center}
$\frac{S}{H}\longrightarrow\frac{G}{H}\longrightarrow\mathbb{C}^*$
\end{center}
the fibration is cohomologically trivial if and only if $H^*\big( \frac{G}{H}\big)\approx H^*\big(\frac{G}{S^0\cap H}\big)$.
\end{coro}
\begin{proof}\renewcommand{\qedsymbol}{$\blacksquare$}
We have the following commutative diagram
\begin{center}
\begin{tikzcd}
\frac{S}{H}\arrow{r}{} & \frac{G}{H}\arrow{r}{\mathcal{X}} & \mathbb{C}^*\\
\frac{S^0}{S^0\cap H}\arrow{r}{} \arrow[swap]{u}{\approx} &\frac{G}{S^0\cap H} \arrow{r}{\tilde{\mathcal{X}}} \arrow[swap]{u}{} & \mathbb{C}^* \arrow{u}
\end{tikzcd}
\end{center}
As $\frac{S^0}{S^0\cap H}$ is connected  the lower fibration is cohomologically trivial by theorem \ref{thm1}. Thus $\frac{S}{H}\rightarrow\frac{G}{ H}\rightarrow\mathbb{C}^*$ is cohomologically trivial if and only if $H^* \big( \frac{G}{H} \big) \approx H^*\big( \frac{G}{S^0 \cap H} \big)$.
\end{proof}
\par
Going back to the example \ref{eg3}, by computation of the cohomologies of the homogeneous spaces $\frac{GL_n}{SO_n}$ and $\frac{GL_n}{O_n}$ shows that the above criterion in corollary \ref{cor1} is satisfied by the fibration 
$\frac{SL_n}{SO_n}\rightarrow\frac{GL_n}{O_n}\rightarrow\mathbb{C}^*$ when $n= 2m+1$ and the fibration does not satisfy the criterion when $n=2m$. \par
We can deduce the general case for an arbitrary, but connected linear algebraic group, from this as follows:

\begin{coro}
Let $\frac{G}{H}$ be a homogeneous space of an arbitrary connected affine algebraic group with a nowhere vanishing function $\mathcal{X}$ to $\mathbb{C}$. If the kernel of the associated character is connected then the fibration is cohomologically trivial.
\end{coro}
\begin{proof}\renewcommand{\qedsymbol}{$\blacksquare$}

Since we are looking at the singular cohomology of the homogeneous spaces, G and $\frac{G}{G_u}$, where $G_u$ is the unipotent radical, are of the same homotopy type. This is because $G_u$
is isomorphic to an affine $n$-space  $\mathbb{A}^n$ (see \cite{KMT}, this is in fact true for any unipotent group defined over a perfect field) which is contractible. Thus it suffices to look at the reductive quotient $\frac{G}{G_u}$. 
Since a subgroup of a unipotent group is unipotent the fibration
\begin{center}
\begin{tikzcd}
\frac{S}{H}\arrow[r] & \frac{G}{H}\arrow[r] & \mathbb{C}^{\ast}
\end{tikzcd}
\end{center}
reduces to 
\begin{center}
\begin{tikzcd}
\frac{S/S\cap G_u}{H/H\cap G_u} \arrow[r] & \frac{G/G_u}{H/H\cap G_u} \arrow[r] & \mathbb{C}^\ast
\end{tikzcd}
\end{center}
\end{proof}

Note that in our analysis quotienting out the unipotent radical is essential. For example, not all characters of $G$ are central when $G$ is solvable, i.e. given a character $\mathcal{X}:G\to\mathbb{C}^*$ the central torus of $Z(G)^0$ may not surject onto $\mathbb{C}^*$. If $G$ is a connected solvable group then the unipotent radical $U$consists of all unipotent elements of $G$ and the quotient $\frac{G}{U}$ is a torus and $G$ is homotopic to $\frac{G}{U}$, which is homotopic to a maximal torus $T$.
Thus consider an exact sequence of a solvable group
 $ S \rightarrow G\xrightarrow{\mathcal{X}} \mathbb{C}^*$. Note that $S\subset G$ is also solvable and $G$ is homotopic to $T$ and similarly $S$ is homotopic to $T_1$, where $T$ and $T_1$ are maximal tori in $G$ and $S$ respectively. Then the above  fibration reduces to $ T_1 \rightarrow T \xrightarrow { \mathcal{X}' }  \mathbb{C}^*$, a fibration of tori where cohomological triviality is equivalent to checking the primitivity of $\mathcal{X}'$.\par
We now consider the case of fibration of homogeneous space of solvable groups 
$\frac{S}{H}\rightarrow \frac{G}{H}\xrightarrow{\mathcal{X}}\mathbb{C}^*$.
We first analyse the exact sequence $H\rightarrow G\rightarrow \frac{G}{H}$.
As before, $H\approx T_2$ and $G\approx T$ where $T_2^0$ and $T$ are maximal tori of $H$ and $G$ respectively. Thus the exact sequence is equivalent to $T_2\rightarrow T\rightarrow \frac{T}{T_2}$. For the same reason the sequence $H\rightarrow S \rightarrow\frac{S}{H}$ is equivalent to $T_2\rightarrow T_1 \rightarrow \frac{T_1}{T_2}$ where $T_1^0$ is a maximal torus in $S$.
Thus the fibration $\frac{S}{H}\rightarrow \frac{G}{H}\xrightarrow{\mathcal{X}}\mathbb{C}^*$ is equivalent to the fibration 
\begin{equation}\label{eq1}
\frac{T_1}{T_2}\rightarrow\frac{T}{T_2}\rightarrow\mathbb{C}^*
\end{equation}

Since we have that $\frac{S}{H}$ is connected the fibre $\frac{T_1}{T_2}$ is also connected and is homeomorphic to $\frac{T_1^0}{T_2^0}$ which is a torus. Thus the fibration (\ref{eq1}) is equivalent to a fibration of tori
\begin{equation*}\label{eq2}
\frac{T_1^0}{T_1^0\cap T_2}\rightarrow\frac{T^0}{T_2^0}\rightarrow\mathbb{C}^*
\end{equation*}
which is cohomologically trivial.

\section{Converse}\label{con}
Let us consider an example.

Let $G$ be a reductive group and $T\subset N(T)$, be a maximal torus of $G$ in the normalizer of $T$. It is a well known result that $H^{\ast}\big(\frac{G}{T}\big)=\mathbb{Q}[W]$, where $W=\frac{N(T)}{T}$ is the Weyl group of $G$ (the Weyl group is a finite group and the cohomology ring of $\frac{G}{T}$ is a regular representation of the Weyl group). Further $H^{\ast}\big(\frac{G}{N(T)}\big)=\mathbb{Q}$. 
Moreover for any subgroup $H$ such that, $T\subset H\subset N(T)$, the cohomology of $\frac{G}{H}$ can be computed by the covering $\frac{H}{T}\rightarrow\frac{G}{T}\rightarrow\frac{G}{H}$. $H^*\big( \frac{G}{H}\big)= H^*\big( \frac{G}{T}\big)^{\frac{H}{T}}$, as graded algebras since $\frac{H}{T}$ is finite. $H^*\big( \frac{G}{T}\big)^{\frac{H}{T}}=\mathbb{Q}[\frac{W}{W_H}]$ where $W_H$ is the image of $H$ in $W$. Motivated by this example we will now construct a class of fibrations that are not cohomologically trivial.

\begin{theorem}
Let $S$ be a reductive group and suppose that $H'\subset H$ be subgroups of $S$ such that $\frac{H}{H'}$ is a finite cyclic group of order $d$ and $H^\ast\big(\frac{S}{H}\big)\xhookrightarrow{\not\approx} H^\ast\big(\frac{S}{H'}\big)$ i.e. $b_j\big(\frac{S}{H}\big)\neq b_j\big(\frac{S}{H'}\big)$ for some $j$.
Then there is an embedding of $H$ in %cover of G, 
 group $S\times\mathbb{C}^*=G$, such that the fibration $\frac{S}{H'}\rightarrow\frac{G}{H}\rightarrow\mathbb{C}^*$ is not cohomologically trivial.
\end{theorem}
\begin{proof}\renewcommand{\qedsymbol}{$\blacksquare$}
Embed $H$ into $G= S\times \mathbb{C}^*$ as follows $H\xrightarrow{(id,\eta)} G$ defined as $h\mapsto (h, hH')$, where $\eta: \frac{H}{H'}\to\mathbb{C}^*$ is an embedding. The image of $H$  has the same number of connected components as $H$, so we will denote the image of $H$ in $G$ as $H$. 
We observe that $\frac{\mathbb{Z}}{<d>}\times \frac{\mathbb{Z}}{<d>}$ acts on $\frac{S}{H'}\times \mathbb{C}^*$ with the quotient $\frac{S}{H}\times\mathbb{C}^*$. 
The diagonal $\frac{\mathbb{Z}}{<d>}=:\Gamma$ in $\frac{\mathbb{Z}}{<d>}\times \frac{\mathbb{Z}}{<d>}$ acts on $\frac{S}{H'}\times\mathbb{C}^*= \frac{S\times\mathbb{C}^*}{H'}$, the action is defined as $\forall g\in \frac{H}{H'}$, $g(sh,t) = (shg, t^{\eta(g)})$ $=(sgh',t^{\eta(g)})$ for some $h,h'\in H', s\in S$, this is the same as the quotienting by the subgroup $H$ embedded in $G$ and hence the quotient is $\frac{G}{\bar{H}}$. Now consider the commutative diagram:
\begin{center}
\begin{tikzcd}
 & {G}\arrow{r}{f\circ\pi}\arrow[swap]{d}{} & \mathbb{C}^*\\
 \frac{S}{H'}\arrow{r}{} & \frac{G}{H}\arrow{ur}{} & 
\end{tikzcd}
\end{center}
where $\pi$ is the second projection and $f:\mathbb{C}^*\to\mathbb{C}^*$ is multiplication by $d$.
We note that the fibre can be identified with $\frac{S}{H'}$, for $H\cap S\times 1 = H'$.\par
We claim that the fibration 
\begin{tikzcd}
 \frac{S}{H'}\arrow{r}{} & \frac{G}{H}\arrow{r}{} & \mathbb{C}^* 
\end{tikzcd}
is not cohomologically trivial. Note that the cohomology of $\frac{G}{H}$ is the cohomology of $\frac{S}{H'}\times \mathbb{C}^*$ invariant under the action of $\Gamma$ i.e. $H^n(\frac{G}{H}) = \big(H^n(\frac{S}{H'}\times\mathbb{C}^*)\big)^\Gamma$. By the K\"unneth formula we have $\big(H^n(\frac{S}{H'}\times\mathbb{C}^*)\big)^\Gamma = \big(\big(H^n( \frac{S}{H'}) \otimes H^0( \mathbb{C}^* )\big)\oplus \big(H^{n-1}(\frac{S}{H'}\otimes H^1( \mathbb{C}^*) \big)^\Gamma= \big(H^n( \frac{S}{H'}) \otimes H^0( \mathbb{C}^* )\big)^\Gamma \oplus \big(H^{n-1}(\frac{S}{H'}\otimes H^1( \mathbb{C}^*) \big)^\Gamma$. Now the action of $\Gamma$ on each graded piece is as follows, $\sigma(a\otimes b)=\sigma(a)\otimes\sigma(b)$, where $\sigma\in  \Gamma, a\in H^i(\frac{S}{H'}), b\in H^j(\mathbb{C}^*) $, further $\sigma(b)=b$ as we are considering cohomology with coefficients in $\mathbb{Q}$ and thus multiplication by $d$ is an isomorphism for $H^1(\mathbb{C}^*)$. Thus if $a\otimes b$ is an invariant cocycle then $\sigma(a)\otimes\sigma(b)=a\otimes b$ i.e if and only if $\sigma(a)=a$. However $H^i(\frac{S}{H'})^\Gamma\approx H^i\big(\frac{\frac{S}{H'}}{\Gamma}\big)\approx H^i(\frac{S}{H})$. Thus by assumption the invariant cohomology is not the entire cohomology and hence the fibration 
\begin{tikzcd}
 \frac{S}{H'}\arrow{r}{} & \frac{G}{H}\arrow{r}{} & \mathbb{C}^* 
\end{tikzcd} 
is not cohomologically trivial.
\end{proof}

Going back to the discussion in the beginning of the section, we had shown that for any subgroup $H$ of $S$ such that $T\subset H\subset N(T)$ and $\frac{H}{T}$ is a cyclic group, then there is a change in the cohomology of $\frac{S}{H}$ and $\frac{S}{T}$ and thus by the theorem we construct $G$ so that $\frac{G}{H}\rightarrow\mathbb{C}^\ast$ is not a cohomologically trivial fibration.

\section{On arbitrary algebraic groups}\label{arb}
We prove cohomological triviality for character maps of homogeneous spaces of arbitrary connected algebraic groups, not necessarily of affine. We will use the notations and conventions of \cite{Brion}, however we continue to work over the field of complex numbers.

To apply our methods we have to find a cental $\mathbb{C}^*$ which surjects onto $\mathbb{C}^*$. This will be possible after we quotient out the unipotent radical, for in a solvable group there are characters which are not central. We will first prove the existence of the unipotent radical in an algebraic group $G$ which is not necessarily affine (For the existence of the unipotent radical we also refer to paragraph 8.41 of \cite{Milne}, for commutative algebraic groups see \cite{Brion1}, theorem 2.9). After quotienting out the unipotent radical we will show there is a central torus which maps surjectively  onto $\mathbb{C}^*$ under a non- trivial character. 

\begin{lem}
Let $G$ be an arbitrary algebraic group. Then there is maximal unipotent normal affine subgroup $U$ of $G$ called as the unipotent radical of $G$.
\end{lem}
\begin{proof}\renewcommand{\qedsymbol}{$\blacksquare$}
Consider the Albanese morphism, $\alpha: G\rightarrow Alb(G)$, where $Alb(G)$ is an abelian variety called as the Albanese variety. Then the kernel of $\alpha$ is a maximal normal  affine algebraic group $G_{aff}$. $G_{aff}$ being affine has a unipotent radical say $G_u$. We claim that $G_u$ is normal in $G$ and is maximal normal unipotent and thus the unipotent radical of G. \par
Since $G_{aff}$ is normal in $G$ we have that conjugate of $G_u$ by any element of $G$ is still in $G_{aff}$ i.e. $g G_u g^{-1}\subset G_{aff}$ $\forall g\in G$. Further more $G_u$ is the unipotent radical of $G_{aff}$ thus $gG_u  g^{-1}$ is connected and unipotent subgroup of $G_{aff}$. Since $G_{aff}$ is normal in $G$ $g_1 g= gg_2$ where $g_1, g_2\in G_{aff}$, therefore $g_1gG_u g^{-1}g_1^{-1}= gg_2G_ug_2^{-1}g^{-1}=gG_ug^{-1}$. Thus $gG_ug^{-1}$ is normal in $G_{aff}$  and hence $gG_u  g^{-1}\subset G_u$. And since the dimensions of $G_u$ and $gG_u  g^{-1}$ are the same we have that $gG_u  g^{-1}= G_u$. Thus $G_u$ is normal in $G$.\par
To see that $Alb(G)$ has no unipotent subgroup we refer to lemma 2.3 in B. Conrad, \textit{Units on Product Varieties}, which says that there are no non constant maps of k- varieties from a linear algebraic group to an abelian variety.
 \end{proof}
 
We thus have the following definition
\begin{defn} \hfill
\begin{enumerate}[label=(\alph*)]%use the package \usepackage{enumitem}
\item An algebraic group $G$ is $\mathit{quasi- reductive}$ if $G_{aff}$ if reductive.
\item An algebraic group $G$ is called $\mathit{anti}$-$\mathit{affine}$ if $\mathcal{O}  \left( G \right)= k$.
 \end{enumerate}
\end{defn}
This definition in particular says that an anti- affine group has no non- trivial characters, or more generally no non constant map to an affine variety.
We further note that any finite connected cover of an anti- affine group is anti- affine.
\begin{defn}
A $\mathit{semi}$-$\mathit{abelian}$ group is any group that can be obtained as an extension of an abelian variety $A$ by a torus $T$, i.e.
\begin{center}
\begin{tikzcd}
1\arrow{r} &T\arrow{r} & G\arrow{r}{q} &A\arrow{r} &1.
\end{tikzcd}
\end{center}
\end{defn}

\begin{theorem}
For a quasi- reductive group $G$ the central torus $Z(G_{aff})^0$ of $G_{aff}$ is the central torus in $G$.
\end{theorem}
\begin{proof}\renewcommand{\qedsymbol}{$\blacksquare$}
By theorem 5.1.1 in \cite{Brion} an algebraic group $G$ is generated by $G_{aff}$ and $G_{ant}$. Further more, $G_{ant}\subset Z(G)$ the centre of $G$(proposition 3.3.5 in \cite{Brion}). cThus $Z(G_{aff})$ commutes with $G_{ant}$ and with $G_{aff}$ and thus with G. Thus $Z(G_{aff})$ is contained in $Z(G)$.\par
Let $T$ be the central torus in $G$.
$G_{aff}$ contains all the connected affine algebraic subgroups and hence $T\subset Z(G_{aff})^0$. Thus the central torus $Z(G_{aff})^0$ of $G_{aff}$ is the central torus in $G$.
\end{proof}

We will now show that the restriction of any character, defined on a quasi- reductive algebraic group, to the central torus of $G_{aff}$ surjects onto $\mathbb{C}^*$. After this cohomological triviality of fibration of homogeneous spaces follows by the same methods as in section 3, for we have never used that $G$ is affine in the section.
\begin{theorem}
Let $G$ be a quasi- reductive group and $\mathcal{X}:G\rightarrow\mathbb{C}^*$ be a character. Then $\mathcal{X}|_{Z(G_{aff})^0}\rightarrow\mathbb{C}^*$ is surjective.    
\end{theorem}
\begin{proof}\renewcommand{\qedsymbol}{$\blacksquare$}
Any non- trivial character $\mathcal{X}$ on $G$ factors through $\frac{G}{G_{ant}}$ as $\mathcal{X}|_{G_{ant}}$ is trivial and $G_{aff}$ surjects onto $\frac{G}{G_{ant}}$.
\begin{center}
\begin{tikzcd}
                  &                                                           &                     &\mathbb{C}^*\\ 
1\arrow{r} & G_{ant} \arrow{r} & G \arrow{r}\arrow[swap]{ur} & \frac{G}{G_{ant}}\arrow{u}\\
 &     Z(G_{aff})^0\arrow{r}& G_{aff}\arrow{u} \arrow[swap]{ur} &
\end{tikzcd}
\end{center}
The restriction of $\mathcal{X}$ to $G_{aff}$ gives a character on $G_{aff}$
and since $G_{aff}$ is reductive $Z(G_{aff})^0$ surjects onto $\mathbb{C}^*$.
\end{proof}

Now the results of section \ref{Ctr} go through without any difference and we have 
\begin{theorem}
Let $G$ be an arbitrary connected algebraic group, and $\mathcal{X}: G\longrightarrow\mathbb{C}^*$ be a non trivial character with the kernel $S$, which we assume to be connected. Then the fibration
\begin{center}
$1\rightarrow S\rightarrow G\rightarrow\mathbb{C}^*\rightarrow 1$
\end{center} 
is cohomologically trivial.
\end{theorem}
\begin{proof}\renewcommand{\qedsymbol}{$\blacksquare$}
$\frac{G}{G_u}$ has the same homotopy type as $G$ we can assume that the group $G$ is quasi- reductive. Since $Z(G_{aff})^0$ surjects onto $\mathbb{C}^*$ we can apply the covering space techniques of proposition \ref{prop1}.
\end{proof}
We also have the analogue of the corollary \ref{coro}
\begin{coro}
A character $\mathcal{X}: G\rightarrow\mathbb{C}^*$, of quasi- reductive groups is split if and only if $\mathcal{X}|_{Z(G)^0}$ is primitive.\hfill $\blacksquare$
\end{coro}

Further, theorem \ref{thm1} is also generalized as
\begin{theorem}
Consider  a fibration of homogeneous spaces of  a connected algebraic group $G$: 
\begin{center}
$\frac{S}{H}\longrightarrow\frac{G}{H}\longrightarrow\mathbb{C}^*$
\end{center}
Suppose that the kernel $S$ of the associated character is connected, then the fibration is cohomologically trivial.\hfill $\blacksquare$
\end{theorem}
We note that this method of study generalizes to a surjective map of quasi- reductive group to a torus and the same results are valid in this context.
In general the question of cohomological triviality is hard. When the base is a commutative group our method of analysis requires an isogenous central subgroup that maps onto the base, which naturally leads to the covering space theory to produce a trivial fibre bundle over the group $G$. This requirement is quite strong: in fact, if the base is an abelian variety then there is no guarantee of an abelian variety in $G$ that maps onto the base. However when the base is the universal semiabelian variety $\frac{G}{D(G_{aff})}$ of a reductive group(c.f. \cite{Brion}, page 51) we have the following:
\begin{theorem}
Let $G$ be a quasi- reductive group and $G_{sab}$ be the quotient semiabelian variety $\frac{G}{D(G_{aff})}$, where $D(G_{aff})$ is the derived group of the affine part of $G$ and also the derived group of $G$ (see corollary 5.1.5 in \cite{Brion}. Then the fibration $$ 1\rightarrow D(G_{aff})\rightarrow G\xrightarrow{\pi} G_{sab}\rightarrow 1$$ is cohomologically trivial.
\end{theorem}
\begin{proof}\renewcommand{\qedsymbol}{$\blacksquare$}
We have the following commutative diagram
\begin{center}
\begin{tikzcd}
1\arrow{r} & D(G_{aff}) \arrow{r} & G\arrow{r}{\pi} & G_{sab} \arrow{r} &1\\
                  &                & G_{ant}\cdot Z(G_{aff})^0  =G' \arrow{u} \arrow[swap]{ur} & & \\
\end{tikzcd}
\end{center}
We note that $G'$ is central in $G$, as $G_{ant}$ and $Z(G_{aff})^0$ are central. Since $G$ is quasi- reductive, $G$ is generated by $D(G_{aff})$, and $G'$, thus $\pi |_{G'}$ is surjective. Further, $ker(\pi|_{G'})= G'\cap D(G_{aff})$ is a central torus contained in $D(G_{aff})$ thus $ker(\pi|_{G'})$ is finite and hence $\pi|_{G'}$ is an isogeny and hence a covering of $G_{sab}$. Now following the method of proof of proposition \ref{prop1} we have the result.
\end{proof}

\begin{remark}
We now give a result for the  structure for arbitrary connected quasi- reductive algebraic group $G$ upto isogeny, as an application of our methods. The existence of the maximal abelian sub variety can already be seen in the work of Rosenlicht corollary page 434 in \cite{Rose}. For a modern treatment we refer to theorem 4.2.5 in \cite{Brion}. \par
First, we make a general comment which we will use: $Q$ be a quotient group of a group $G$, let $H$ be a central sub- group of $G$ that is isogenous to $Q$, consider the fibre product
\begin{center}
\begin{tikzcd}
\tilde{G} \arrow{r}{}\arrow[swap]{d}{} & H\arrow{d}{}\\
G \arrow{r}{q}            & Q          
\end{tikzcd}
\end{center}
By the technique of the proof of proposition  \ref{prop1}, we have $\tilde{G}\approx H\times ker(q)$. Thus it is enough to find a central subgroup that is isogenous to a quotient group to get such a decomposition of $G$ upto isogeny. \par
Consider the morphism $ q : G\to  \frac{G}{G_{ant}} \to T'$, where $T'$ is the maximal toral quotient of the group $\frac{G}{G_{ant}}$(and hence a maximal toral quotient of $G$). Now the center $Z(G_{aff})^0$ surjects onto $T'$ and thus there is a torus $T\subset Z(G_{aff})^0$ that is isogenous to $T'$. Hence by the earlier the comment in the previous paragraph we have an isogeny of $G\approx T\times G_1$, where $G_1$ is the kernel of the composition. \par
We note that the derived groups of $G$ and $G_1$ are the same, since the derived groups are affine, as $G_{ant}$ is commutative $G_{ant}$ is trivial in $D(G)$ (also corollary 5.1.5 \cite{Brion} and since $G_{aff}$ is reductive $D(G)$ is semisimple ) and $G_{ant}\subset G_1$. Now consider the  morphism $f : G_1 \to \frac{G_1}{D(G)}=Q$, $G_{ant}$ maps surjectively onto $Q$ and the kernel $G_{ant}\cap D(G_1)$ is finite (the affine part of $G_{ant}$ is a torus!). Thus $G_{ant}$ is isogenous to $Q$ and hence $G_1\approx D(G) \times G_{ant}$.\par

Let $A$ be a maximal abelian subvariety of $G_{ant}$ and consider the image of $A$, say $A'$ under the albanese morphism $\alpha$. By Poincar\'e's complete reducibility theorem (c.f. theorem 1 page 160 \cite{mum}) $A'$ has a complementary abelian subvariety in $\alpha(G_{ant})$ say $B$ upto isogeny. Taking the composition of $\alpha$ with the quotient map to $\frac{\alpha(G_{ant})}{B}\approx A'$ we have a map $q:G_{ant}\to A' $ and the restriction of  $q$ to $A$ is an isogeny (the kernel is $((G _{ant})_{aff}\cap A$ which is finite). Thus $G_{ant}\approx A\times G_S$, where $G_S$ is a semi- abelian, anti- affine variety that has no proper abelian subvarieties. However, we note that the albanese map of $G_S$ is non trivial. Further, $A$ is the unique maximal abelian variety in $G_{ant}$, this follows as $A$ is the unique abelian variety in $\tilde{G}_{ant}$. \par
Hence upto isogeny, any quasi- reductive algebraic group has the following structure $G_{ss}\times T \times A \times G_S$, where $G_{ss}$ is affine semisimple, $T$ is a torus, $A$ is the unique abelian variety that is split (in $ G_{ant}$) and $G_S$ is a semi- abelian, anti- affine variety with no proper abelian subvariety. \par 
Further, consider a homogeneous space $\frac{G}{H}$ with surjective morphism onto a commutative algebraic group $B$, so that the composition with the quotient map $q: G\to\frac{G}{H}$ is a homomorphism of algebraic groups, with kernel $S$ connected. Now suppose that, there is a central subgroup $G_1\subset G$ which is isogenous to $B$, then the cover of $\frac{G}{H}$ is $\tilde{\frac{G}{H}}\approx\frac{S}{H}\times G_1$, where $\frac{S}{H}$ is the fibre over $e$ and the fibration $\frac{S}{H}\to\frac{G}{H}\to B$ is cohomologically trivial. 
\end{remark}

\end{document}